\documentclass[11pt]{article}

\usepackage{amsmath}
\usepackage{latexsym}
\usepackage{amsfonts}
\usepackage{amssymb}
\usepackage{amscd}
\usepackage{theorem}
\usepackage{color}

\newtheorem{theorem}{Theorem}[section]

\newtheorem{lemma}[theorem]{Lemma}
\newtheorem{proposition}[theorem]{Proposition}

{\theorembodyfont{\rmfamily}

}

\newenvironment{proof}{	   % De proof o de prueba seg\'{u}n convenga
  \noindent
  \textbf{Proof.}}{
  \hfill $\Box$
  \vspace{3mm}
}

\numberwithin{equation}{section}

\newcommand{\N}{\mathbb{N}} %% Conjunto naturales:	   \N
\newcommand{\C}{\mathbb{C}} %% Conjunto complejos:	   \C
\newcommand{\D}{\mathbb{D}} %% Disco unidad:		   \D

\def\({\left(}       \def\){\right)}

\begin{document}

\title{Monomial basis in Korenblum type spaces of analytic functions. }

\author{Jos\'e Bonet, Wolfgang Lusky and Jari Taskinen}

\date{}

\maketitle

\begin{abstract}
It is shown that the monomials $\Lambda=(z^n)_{n=0}^{\infty}$ are a Schauder basis of the Fr\'echet spaces $A_+^{-\gamma}, \ \gamma \geq 0,$ that consists  of all the analytic functions $f$ on the unit disc such that $(1-|z|)^{\mu}|f(z)|$ is bounded for all $\mu > \gamma$. Lusky \cite{L} proved that $\Lambda$ is not a Schauder basis for the closure of the polynomials in weighted Banach spaces of analytic functions of type $H^{\infty}$. A sequence space representation of the Fr\'echet space $A_+^{-\gamma}$ is presented. The case of (LB)-spaces $A_{-}^{-\gamma}, \ \gamma > 0,$ that are defined as unions of weighted Banach spaces is also studied.
\end{abstract}

%% Footnotes
\renewcommand{\thefootnote}{}

%%%%%%%%%%%%%%%%%%%%%%%%%%%%%%%%%%%%%%%%%%%%%%%%%%%%%%%%%%%%%%%%%%%%%% %%%%%%%%%% AQU\'{I} EMPIEZA EL DOCUMENTO %%%%%%%%%%%%%%%%%%%%%%%%%%%%% %%%%%%%%%%%%%%%%%%%%%%%%%%%%%%%%%%%%%%%%%%%%%%%%%%%%%%%%%%%%%%%%%%%%%%
%%%%%%%%%%%%%%%%%%%%%%%%%%%%%%%%%%%%%%%%%%%%%%%%%%%%%%%%%%%%%%%%%%%%%%
%%%%%%%%%%
\section{Introduction and preliminaries}
%%%%%%%%%%

We consider analytic functions $f \in H(\D)$ on the unit complex disc $\mathbb{D}=
	\{ z \in \mathbb{C} : |z| <1 \}$. For a function $f:\mathbb{D} \rightarrow \mathbb{C}$ and $0 \leq r < 1 $ we put $M_{\infty}(f,r) = \sup_{|z|=r}|f(z)|$. If $f$ is analytic then $M_{\infty}(f,r)$ is increasing with respect to $r$. For
	$\mu > 0$ let
	\[ ||f||_{\mu} = \sup_{0 \leq r < 1}M_{\infty}(f,r)(1-r)^{\mu} \]
	and $A^{-\mu}= \{ f:\mathbb{D} \rightarrow \mathbb{C} : f \mbox{ analytic }, ||f||_{\mu} < \infty \}$. Moreover let
	\[ A^{-\mu}_0 = \{ f \in A^{-\mu} : \lim_{r \rightarrow 1}M_{\infty}(f,r)(1-r)^{\mu} =0 \} \]
	and for $\gamma \in [0,\infty[$
	\[A^{-\gamma}_+ = \cap_{ \mu > \gamma} A^{-\mu} =\cap_{ \mu > \gamma} A^{-\mu}_0. \]
	We consider the norms $|| \cdot ||_{\mu}$, $\mu > \gamma$, with which $ A^{-\gamma}_+$ becomes a Frechet space. By definition we have
	\[  || \cdot ||_{\mu_1} \leq || \cdot ||_{\mu_2} \ \ \mbox{ and } A^{-\mu_2} \subset
	A^{-\mu_1} \ \ \mbox{   whenever }
	\mu_1 > \mu_2. \]
	Similarly, for $\gamma \in ]0,\infty]$, let
	\[A^{-\gamma}_- := \cup_{ \mu < \gamma} A^{-\mu} =\cup_{ \mu < \gamma} A^{-\mu}_0 \]	
	be endowed with the finest locally convex topology such that all inclusions $A^{-\mu} \subset A^{-\gamma}_-$ 	are continuous. With this topology $A^{-\gamma}_-$ is an (LB)-space, i.e.\ a Hausdorff
countable inductive limit of Banach spaces.
	
The \textit{Korenblum space} $A_-^{-\infty}$, denoted simply by $ A^{-\infty}$ \cite{Kor},  is defined via
\[
A^{-\infty}:=\cup_{0<\gamma<\infty}A^{-\gamma}=\cup_{n\in\N}A^{-n}.
\]

Spaces of this type play a relevant role in interpolation and sampling of analytic functions, see \cite{HKZ}.
Weighted spaces of analytic functions appear in the study of growth conditions of
analytic functions and have been investigated in various articles since
the work of Shields and Williams, see  e.g.\  \cite{BBG},\cite{BBT},
\cite{L}, \cite{SW} and the references therein.

Our notation for functional analysis is standard; see e.g.\ \cite{MV}. We recall that 
a sequence $(x_n)_n$ in a locally convex space $E$ is a \textit{Schauder basis} if every element $x \in E$ can be written in a unique way as $x = \sum_{n=1}^{\infty} u_n(x) x_n$ with $u_n: E \rightarrow \mathbb{K}, n \in \N,$ continuous linear forms.
We refer the reader to \cite{LT} for more information about Schauder bases in Banach spaces and to \cite{J} for Schauder bases on locally convex spaces. 

Let $e_n(z) = z^n$, $z \in \mathbb{D}$, for $n=0,1,2, \ldots$ and $\Lambda = \{ e_n : n=0,1,2, \ldots\}$. The second author proved in \cite{L} that $\Lambda$ is not a Schauder basis for any $ A^{-\mu}_0 $ and in more general weighted Banach spaces of analytic functions.
On the other hand,  the monomials $(e^n)_n$ constitute a Schauder basis of the space $A^{- \infty}$.
In fact associating each $f(z) = \sum_{n=0}^{\infty} a_n z^n \in A^{-\infty}$ to the sequence $(a_n)_n$
of Taylor coefficients defines a linear topological isomorphism from $A^{- \infty}$ into the strong dual $s'$ of the Fr\'echet echelon space $s$ of rapidly decreasing sequences.

The purpose of this note is to answer the following two questions:

\textbf{Question 1:} Are the monomials a Schauder basis of the spaces $A_+^{-\gamma}$ and $A_-^{-\gamma}$ for $\gamma \neq \infty$?

\textbf{Question 2:} Are there sequence space representations of the spaces $A_+ ^{- \gamma}$ for $0 \le \gamma < \infty $, (resp.\ $A_- ^{- \gamma}$, for $0 < \gamma < \infty $) as K\"othe echelon (resp.\ K\"othe co-echelon) spaces of order $0$?

In connection with question 2, recall that  the Banach spaces $A_0^{-\mu}$ and $A^{-\mu}$ are isomorphic to $c_0$ and $\ell_{\infty}$ respectively \cite{SW}, although the monomials are not a Schauder basis of them \cite{L}.

Question 1 is answered positively in Theorem \ref{mainbasis} and question 2 is dealt with in Section \ref{sequence}; see Theorem \ref{sequence_space_rep}.

\section{ Monomial bases}

The following lemma is easy to prove.

\begin{lemma}\label{max}
Let $\mu > 0$ and $N>0$. The function $r^N(1-r)^{\mu}$, $0 \leq r \leq 1$ has a global maximum point at $r$ if and only if $N = \mu r(1-r)^{-1}$.
\end{lemma}

For $n > \mu>0$ put $\rho_{n,\mu} = 1 - \frac{\mu}{n}$. Then $\rho_{n,\mu}$ is the global maximum point of $r^{n- \mu}(1-r)^{\mu}$.

\begin{lemma} \label{lemma_basis}
Let $n \in \mathbb{N}$, $n > \mu$. Consider  $f:\mathbb{D} \rightarrow \mathbb{C}$ analytic with
		$ f(z) = \sum_{k=n}^{\infty} a_k z^k$. Then
		\[ ||f||_{\mu} = \sup_{\rho_{n,\mu} \leq r < 1} M_{\infty}(f,r)(1-r)^{\mu}.  \]
\end{lemma}
\begin{proof}
Let $g(z)= z^{-n}f(z)$. Then,  $g$ can be regarded as analytic function on $ \mathbb{D}$ (with the
	natural extension to $0$). We obtain, for $0 \leq r < 	\rho_{n,\mu}$,
	\begin{eqnarray*}
	M_{\infty}(f,r)(1-r)^{\mu} & =     &  r^n M_{\infty}(g,r)(1-r)^{\mu} \\
	                           & \leq  &  \left(\frac{r}{\rho_{n,\mu}}\right)^n \left(\frac{1-r}{1-\rho_{n,\mu}}\right)^{\mu}
	                        \rho_{n,\mu}^n   M_{\infty}(g,\rho_{n,\mu})(1-\rho_{n,\mu})^{\mu}  \\
                               & \leq 	 &
 \left(\frac{r}{\rho_{n,\mu}}\right)^{n- \mu} \left(\frac{1-r}{1-\rho_{n,\mu}}\right)^{\mu}
 \rho_{n,\mu}^n   M_{\infty}(g,\rho_{n,\mu})(1-\rho_{n,\mu})^{\mu}  \\
                               & \leq  & M_{\infty}(f,\rho_{n,\mu})(1-\rho_{n,\mu})^{\mu} ,
                        			\end{eqnarray*}
where we have used  the fact that $\rho_{n,\mu}$ is the global maximum point of $r^{n- \mu}(1-r)^{\mu}$.
\end{proof}

\begin{proposition} \label{prop_basis}
Let  $\mu_0 > 0$ and  $\mu > \mu_0$.
	 Then, for any $f \in  A^{-\mu_0}$ the Taylor series of $f$ converges
		to $f$ with respect to $|| \cdot ||_{\mu}$.
\end{proposition}
\begin{proof}
Let $P_n$ be the Dirichlet projections, i.e. $P_nf$ is the $n$'th partial sum of the Taylor series of $f$. It is well known that there is  a universal constant $c > 0$ such that for every analytic function $f$, every $n$ and every radius $r$ have
                        		\[  M_{\infty}(P_nf,r) \leq c \log(n)  M_{\infty}(f,r).  \]
See e.g.\ \cite{T}.

We obtain, for $f \in A^{-\mu_0}$,
      \[ ||f-P_nf ||_{\mu_0} \leq c (1+\log(n)) ||f||_{\mu_0}. \]
If $f(z) =\sum_{k=0}^{\infty} a_k z^k$ then $(id-P_n)f(z)=\sum_{k=n+1}^{\infty} a_k z^k$. For $ \mu > \mu_0$ we
apply Lemma \ref{lemma_basis} to get
      \begin{eqnarray*}
      	||(id-P_n)f||_{\mu} & =       & \sup_{\rho_{n+1,\mu} \leq r < 1} M_{\infty}((id-P_n)f,r)(1-r)^{\mu}  \\
      	                    & \leq    & \sup_{\rho_{n+1,\mu} \leq r < 1} (1-r)^{\mu - \mu_0} ||(id-P_n)f||_{\mu_0} \\
      	                    & \leq    & (1-\rho_{n+1,\mu})^{\mu - \mu_0}(1+ \log(n))||f||_{\mu_0} \\
      	                    & =       & \left(\frac{\mu}{n+1} \right)^{\mu - \mu_0}	(1+ \log(n))||f||_{\mu_0}.
      \end{eqnarray*}
  Since $\mu - \mu_0 > 0$ the right-hand side goes to 0 if $n \rightarrow \infty$. This proves the proposition.
\end{proof}

\begin{theorem}\label{mainbasis}
\begin{itemize}
\item[(i)]	$\Lambda$ is a Schauder basis of $A^{-\gamma}_+$ for any $\gamma \geq 0$.

\item[(ii)]	$\Lambda$ is a Schauder basis of $A^{-\gamma}_-$ for any $\gamma >0$.
\end{itemize}
\end{theorem}
\begin{proof}
(i) We have to prove that the Taylor series of every $f \in A^{-\gamma}_+, \gamma \geq 0$ converges in
$A^{-\gamma}_+$ to $f$. Fix $\mu > \gamma$ and select $\mu$ with $\gamma < \mu_1 < \mu$. Since $f \in A^{-\mu_1}$, we can apply Proposition \ref{prop_basis} to conclude that the Taylor series of $f$ converges
in $A^{\mu}$ to $f$. This implies the conclusion.

(ii) is a direct consequence of Proposition \ref{prop_basis} and the properties of inductive limits. \end{proof}

It is well-known that the Korenblum space $A^{- \infty}$   is nuclear, since it is isomorphic to the nuclear (LB)-space $s'$. The following result is proved in \cite{ABR}.

\begin{proposition} \label{pra}
Each  Fr\'echet space  $A_+ ^{- \gamma}$ for $0 \le \gamma < \infty $, and each {\rm(LB)}-space $A_- ^{- \gamma}$, for $0 < \gamma < \infty $, fails to be nuclear.
\end{proposition}

This result is now a direct consequence of Theorem \ref{mainbasis} and Grothendieck Pietsch criterion \cite[Theorem 28.15]{MV}. We indicate the argument for $A_+ ^{- \gamma}$: If this Fr\'echet space is nuclear, given $\mu:= \gamma +1$, we can apply \cite[Theorem 28.15]{MV} to find $\gamma < \nu < \mu$ such that
$\sum_{n=1}^{\infty} \frac{||z^n||_{\mu}}{||z^n||_{\nu}} < \infty$. This implies by Lemma \ref{max} that $\sum_{n=1}^{\infty} \frac{1}{n^{\mu-\nu}} < \infty$. A contradiction, since $0<\mu-\nu<1$.

\section{Sequence space representation} \label{sequence}

We recall the definition of K\"othe echelon and co-echelon spaces of order infinity; see \cite{BMS} and \cite[Chapter 27]{MV}. A sequence  $A = (a_k)_k$ of functions $a_k: \N \cup \{0\} \to ]0, \infty)$  is called a \textit{K\"{o}the matrix} on $\N$ if  $0 < a_k (j) \le a_{k+1} (j)$ for all $j \in \N\cup \{0\}$ and $k\in\N$. The \textit{K\"{o}the echelon space of order infinity} associated to  $A$ is
\[
\lambda_{\infty} (A) : = \{x \in \C^{\N} :  \sup_{j  }a_k (j) x_j < \infty , \ \forall k \in\N\},
\]
which is a Fr\'echet space relative to the increasing sequence of canonical seminorms
$$
q^{(\infty)}_k (x) := \sup_{j } a_k (j ) |x_j |, \quad  x\in \lambda_{\infty} (A), \quad k \in \N.
$$
Then $\lambda_{\infty}(A)=\cap_{k \in \N} \ell_{\infty}(a_k)$. Here $\ell_{\infty}(a_k)$ is the usual weighted $\ell_{\infty}$ sequence space.

Given a decreasing sequence $V=(v_k)_k$ of strictly positive functions on $\N\cup \{0\}$, the K\"othe co-echelon space of order infinity is $k_{\infty}(V):= {\rm ind}_k \ell_{\infty}(v_k)$ and it is endowed with the inductive limit topology. Then $k_{\infty}(V)$ is a regular (LB)-space \cite{BMS}.

Given $\mu \in ]0,\infty[$ define $r_{\mu}(0)=s_{\mu}(0):=1$ and
$$
r_{\mu}(j):= \frac{\mu}{2^n + \mu} \ \ \ j=2^n,...,2^{n+1}-1, \ \ \ n=0,1,2,...
$$
and
$$
s_{\mu}(j):= \frac{\mu}{j + \mu} \ \ \ j=1,2,...
$$

\begin{lemma}\label{lem_decr}
If $0<\mu_2 < \mu_1$, then $r_{\mu_1}(j) \leq r_{\mu_2}(j)$ and $s_{\mu_1}(j) \leq s_{\mu_2}(j)$ for each $j=0,1,2,...$
\end{lemma}
\begin{proof}
It is enough to show that  the function
\[ f(x) = \left( \frac{x}{j + x} \right)^x = \exp\left(x \log(x) -x \log(j +x) \right), \ \ \ x > 0, \]
is decreasing.
It is easily seen that $f'(x) \leq 0$ if and only if
\[ 1 + \log(x) - \log(j+x) - \frac{x}{j+x} \leq 0. \]
This inequality is valid for all $x >0$ since $t \leq e^{t-1}$ for each $t \in ]0,1[$ implies
$$
 \frac{x}{j+x} \leq \exp\left( \frac{x}{j+x} -1 \right)
$$
 for all $ x > 0$.
\end{proof}

Given $\gamma \geq 0$, put $\mu_k:= \gamma + \frac{1}{k}, k \in\N,$ and define $a_k(j):=s_{\mu_k}(j), j=0,1,2,..., k \in \N$. Lemma \ref{lem_decr} implies that $A_{\gamma}:=(a_k)_k$ is a K\"othe matrix. Analogously, for $\gamma > 0$, we set $\nu_k= \gamma - \frac{1}{k}$ with $k$ large enough so that $\nu_k>0$. Now, by Lemma \ref{lem_decr} the sequence $V_{\gamma}:=(v_k)_k$, $v_k(j):= s_{\nu_k}(j), j=0,1,2,..., k \in \N$ is decreasing. Keeping this notation, we can state the main result of this section

\begin{theorem}\label{sequence_space_rep}
\begin{itemize}
\item[(i)]	For each $\gamma \geq 0$ the Fr\'echet space $A^{-\gamma}_+$ is isomorphic to the K\"othe echelon space $\lambda_{\infty}(A_{\gamma})$.

\item[(ii)]	For each $\gamma > 0$ the (LB)-space $A^{-\gamma}_-$ is isomorphic to the K\"othe co-echelon space $k_{\infty}(V_{\gamma})$.
\end{itemize}
\end{theorem}

The \textbf{proof of the Theorem \ref{sequence_space_rep}} is a consequence of the results presented below.

Firstly, we introduce, for a sequence $(x_j)_{j=0}^{\infty}$ of complex numbers, the norms
	\[ |||(x_j)|||_{\mu} = \sup \left(|x_0|,\sup_{n= 0,1,2, \ldots} \left(\frac{\mu}{2^n+\mu}\right)^{\mu} \sup_{2^n \leq j < 2^{n+1}}	|x_j| \right) = \sup_{j} r_{\mu}(j)|x_j|  \]
	 and define
	\[B_{\gamma} =\{ (x_j) : |||(x_j)|||_{\mu} < \infty \ \mbox{ for all } \mu > \gamma \}. \]
	We consider the locally convex topology on $B_{\gamma}$ generated by the norms $||| \cdot|||_{\mu} $ for all
	$\mu > \gamma$. Finally put
	\[C_{\gamma} =\{ (x_j) : |||(x_j)|||_{\mu} < \infty \ \mbox{ for some } \mu < \gamma \} \]
	endowed with the finest locally convex topology such that the embedding $J_{\mu}: \{  (x_j) : |||(x_j)|||_{\mu} < \infty \} \rightarrow C_{\gamma}$ is continuous for all $ \mu < \gamma$.

Since $s_{\mu}(j) \leq r_{\mu}(j) \leq 2^{\max(1,\mu)} s_{\mu}(j)$ for each $j=0,1,2,...$ it follows that
$B_{\gamma} = \lambda_{\infty}(A_{\gamma})$ and $C_{\gamma} = k_{\infty}(V_{\gamma})$ algebraically and topologically. In order to complete the proof of Theorem \ref{sequence_space_rep}, we must show that $A^{-\gamma}_+$ and $B_{\gamma}$, as well as $A^{-\gamma}_-$ and $C_{\gamma}$,   are isomorphic.

To this end, given $f \in H(\D)$
 with $f(z) = \sum_{j=0}^{\infty} a_j z^j$, put $ f_n(z) = \sum_{j=2^n}^{2^{n+1}-1} a_jz^j$.
Define $(Tf)(0)= a_0$ and
\[ (Tf)(j) =  f_n( e^{i2 \pi j/2^n}) \ \ \ \mbox{ if } \ \ 2^n \leq j \leq 2^{n+1}-1
{\color{blue} \ \ \ (*) } \]
and $Tf = (\ (Tf)(j) \ )_{j=0}^{\infty}$.

The following technical result will be proved at the end of this section.

\begin{lemma}\label{technical}
For each $0 < \mu_1 < \mu < \mu_2$ there are constants $d_1 >0$ and $d_2 > 0$ such that the following holds
\begin{itemize}

\item[(i)] $||| Tf |||_{\mu} \leq d_2 ||f||_{\mu_1}$ for every 	$f \in H(\D)$.

\item[(ii)] For each $x=(x_j)$ such that $||| x |||_{\mu} < \infty$ there is $f \in H(\D)$ such that
$Tf = x$ and $d_1||f||_{\mu_2}\leq ||| x |||_{\mu}$.
\end{itemize}
\end{lemma}

\begin{proposition}\label{isomorphism}
\begin{itemize}
\item[(a)]	$T|_{A^{-\gamma}_+}$ is an isomorphism between $A^{-\gamma}_+$ and $B_{\gamma}$.
	
\item[(b)]	$T|_{A^{-\gamma}_-}$ is an isomorphism between $A^{-\gamma}_-$ and $C_{\gamma}$. \\
\end{itemize}
\end{proposition}
\begin{proof}
(a) Lemma \ref{technical} (i) shows that $T$ is well defined and continuous. On the other hand, part (ii) implies that $T$ is bijective. For the injectivity observe that the values
$f_n(e^{i 2 \pi j / 2^n})$ are unique, since $f_n(z) / z^{2^n}$ is a polynomial of degree
at most $2^{n}-1$, and its value is taken at $2^n$ different points. See also the Lemma \ref{technical} below. Finally, the estimate in  Lemma \ref{technical}(ii) shows that $T|_{A^{-\gamma}_+}$ is an isomorphism.
The continuity of the inverse can also be deduced by the open mapping theorem for Fr\'echet spaces.

The proof for (b) is similar.
\end{proof}

Proposition \ref{isomorphism} completes the proof of Theorem \ref{sequence_space_rep}.

It remains  \textbf{to prove Lemma \ref{technical}}. Its proof is technical and requires several steps.

First we recall
some basic facts from  classical approximation theory. See \cite{T} and \cite{W}. Let, for $m \in \mathbb{N}$,
\[ D_m(\varphi) = \sum_{j= -m}^m e^{i j \varphi}, \ \ \ \varphi \in [0, 2 \pi], \]
be the Dirichlet kernel and put
\[ (P_mf)(re^{i \varphi} ) = (D_m \ast f)(re^{i \varphi})= \frac{1}{2 \pi} \int_0^{2 \pi} D_m(\varphi- \psi)
f(re^{i  \psi}) d \psi. \]
Then we obtain
\[  (P_mf)(re^{i \varphi})=\sum_{j= -m}^m a_jr^je^{i j \varphi} \ \ \mbox{ provided that } \ \ \
f(re^{i \varphi})=\sum_{j= -\infty}^{\infty} a_jr^je^{i j \varphi}. \]
Let, for $r > 0$, $ M_1(f,r) = (2 \pi)^{-1} \int_0^{2 \pi} |f(re^{i \varphi})| d \varphi$. It is well-known that
\[ D_m \geq 0, \ \ M_1(D_m,1) \leq c \log( m), \ \ M_q(P_mf,r) \leq c \log( m)M_q(f,r) \]
 if  $ q \in \{ 1, \infty \}$.
Here $c > 0$ is a constant independent of $m$.

The following lemma is essentially known. Since we do not have a precise reference we insert a proof which is
a modification of the proof of \cite[II E 9]{W}.

\begin{lemma}\label{estimates}
There is a universal constant $c > 0$ such that, for any f with
$f(z) = \sum_{j= 2^n}^{2^{n+1}-1} a_j z^j$, we have
\[ \sup_{j=1, \ldots, 2^n} |f(e^{i2 \pi \ j/2^n})| \leq M_{\infty}(f,1) \leq c n^2 \sup_{j=1, \ldots, 2^n} |f(e^{i2 \pi \ j/2^n})|. \]
\end{lemma}
\begin{proof}
Let $\varphi_j = 2 \pi \ j/2^n$, $j=1, \ldots, 2^n$. For functions $g$ of the form $g( \varphi) = \sum_{k=- 2^n}^{2^{n}} b_k \exp(i k \varphi)$ we have, since
$ \sum_{j=1}^{2^n} \exp (i 2 \pi k j /2^n ) =0 $ for $k \not=0$,
\[ (3.1) \hspace{3cm} \frac{1}{2^n} \sum_{j=1}^{2^n} g(\varphi_j)= b_0 = \frac{1}{2 \pi}\int_0^{2 \pi} g(\varphi)
d \varphi. \hspace{6cm} \]
 We claim
 \[ (3.2) \hspace{3cm}\frac{1}{2^n} \sum_{j=1}^{2^n}| g(\varphi_j)| \leq c n  \frac{1}{2 \pi}\int_0^{2 \pi} |g(\varphi)|
 d \varphi  \hspace{5cm} \]
 where $c > 0$ is a universal constant. Indeed, we have $D_{2^n} \ast g = g$ and hence, using (3.1), we conclude
 \begin{eqnarray*}
 	\frac{1}{2^n} \sum_{j=1}^{2^n}| g(\varphi_j)|  & =      & \frac{1}{2^n} \sum_{j=1}^{2^n}
  |\frac{1}{2 \pi}\int_0^{2 \pi} D_{2^n}(\varphi_j - \psi)g(\psi)d \psi|  \\
 	                                              & \leq   &	
 \frac{1}{2 \pi}\int_0^{2 \pi}\frac{1}{2^n} \sum_{j=1}^{2^n} D_{2^n}(\varphi_j - \psi)|g(\psi)|d \psi  \\
                                                  & =      &
 \frac{1}{2 \pi}\int_0^{2 \pi}\frac{1}{2\pi}\int_0^{2^n}D_{2^n}(\varphi -\psi)d\varphi        |g(\psi)| d\psi                                    \\
                                                  & \leq   &
  c n \frac{1}{2 \pi}\int_0^{2 \pi} |g(\psi)|
  d \psi.
 \end{eqnarray*}
Now take $f$ as in the statement and put
\[g(e^{i \varphi}) = e^{-i3 \cdot2^{n-1} \varphi} f(e^{i \varphi}) = \sum_{j=-2^{n-1}}^{2^{n-1}-1} a_{j+3\cdot2^{n-1}}
e^{i j \varphi}. \]
We use that $l \cdot g$ is a trigonometric polynomial of degree $2^n$ if $l$ is a trigonometric polynomial of
degree $2^{n-1}$.

For each $\varepsilon>0$, we choose  $h \in L_1( \partial \mathbb{D})$ such that $M_1(h,1) = 1$ and $ \frac{1}{1 + \varepsilon} M_{\infty}(g,1) \leq  | \int_0^{2 \pi} h(e^{i \varphi})g(e^{i \varphi}) d \varphi|$.
Then, using (3.2), we get
\begin{eqnarray*}
 \frac{1}{1 + \varepsilon} M_{\infty}(f,1) & =    & \frac{1}{1 + \varepsilon} M_{\infty}(g,1) \\
                 & =   & | \int_0^{2 \pi} h(e^{i \varphi})g(e^{i \varphi}) d \varphi| \\
                 & =    &| \int_0^{2 \pi} (D_{2^{n-1}}h)(e^{i \varphi})g(e^{i \varphi}) d \varphi| \\
                 & =    &   |\frac{1}{2^n} \sum_{j=1}^{2^n}(D_{2^{n-1}}h)(e^{i \varphi_j})g(e^{i \varphi_j})| \\
                 & \leq &  \frac{1}{2^n} \sum_{j=1}^{2^n}|(D_{2^{n-1}}h)(e^{i \varphi_j})| \cdot|g(e^{i \varphi_j})| \\
                 & \leq & c n\int_0^{2 \pi} |(D_{2^{n-1}}h)(e^{i \varphi})| d \varphi  \sup_j |g(e^{i \varphi_j})| \\
                 & \leq & c^2n^2 M_1(h,1)\sup_j |g(e^{i \varphi_j})| \\
                 & =    & c^2n^2\sup_j |f(e^{i \varphi_j})|,
                 \end{eqnarray*}
                 
where the third equality follows from
the restriction of the degree of $g$ and the usual orthonormality relations.

Since $\varepsilon$ is arbitrary, this proves the right-hand side inequality of the statement. The left-hand side is trivial.
\end{proof}

\textbf{Completion of the proof of Lemma \ref{technical}.}
   We consider $r_{\mu,n}= 1 - \mu/(2^n+ \mu)$ for given $\mu > 0$. The function
$r^{2^n}(1-r)^{\mu}$ attains its maximum at  $r_{\mu,n}$. Let
$f(z) = \sum_{j=0}^{\infty} a_j z^j \in H(\D)$ and  $ f_n(z) = \sum_{j=2^n}^{2^{n+1}-1} a_jz^j$.
It suffices to consider the case $f(0) = a_0=0$. Put $g_n (z) =  \sum_{j=0}^{2^n-1}
a_{j+2^n}z^j$. 
We obtain, for $r < r_{\mu,n}$,
\begin{eqnarray*}
%%(3.3) \hspace{1cm}
M_{\infty}(f_n,r) (1-r)^{\mu} & \leq   &   \frac{r^{2^n} (1-r)^{\mu}}{r_{\mu,n}^{2^n}(1-
r_{\mu,n})^{\mu}} M_{\infty}(g_n,r) r_{\mu,n}^{2^n}(1-r_{\mu,n})^{\mu} \\
 & \leq &   M_{\infty}(g_n,r_{\mu,n}) r_{\mu,n}^{2^n}(1-r_{\mu,n})^{\mu}
 \\   & \leq  &
M_{\infty}(f_n,r_{\mu,n})(1-r_{\mu,n})^{\mu} \\
 & \leq  &     M_{\infty}(f_n,1 )(1-r_{\mu,n})^{\mu}.
\end{eqnarray*}
We have for  $r_{\mu,n} < s < 1$,
 \begin{eqnarray*}
M_{\infty}(f_n,s) (1-s)^{\mu}  \leq
M_{\infty}(f_n,1) (1-r_{\mu,n})^{\mu}
\end{eqnarray*}
and combining this with the previous estimate yields
\[ (3.3) \hspace{1cm} ||f_n||_{\mu}
\leq  M_{\infty}(f_n,1)(1-r_{\mu,n})^{\mu}. \hspace{5cm}
\]

Moreover we have, by \cite[Lemma 3.1.(a)]{L}, %% if $r_{\mu,n} < s < 1$,
\begin{eqnarray*}
(3.4) \hspace{1cm}
M_{\infty}(f_n,1)
& \leq & \left(\frac{1}{r_{\mu,n}} \right)^{2^{n+1}}
M_{\infty}(f_n,r_{\mu,n})  \\
&=& (1+ \frac{\mu}{2^n})^{2^{n+1}}
M_{\infty}(f_n,r_{\mu,n})\\
&\leq &
c_1  M_{\infty}(f_n,r_{\mu,n})
 \end{eqnarray*}
for a universal constant $c_1$.  	

Now let $ \mu_1 < \mu < \mu_2$. In view of (3.4) we have
 	\begin{eqnarray*}
 \lefteqn{ \sup_n\frac{\mu^{\mu}}{(2^n+\mu)^{\mu}}
  \sup_{2^n \leq j < 2^{n+1}}|f_n(e^{i 2 \pi j/2^n})| } \\
  & &  \leq  \sup_n\frac{\mu^{\mu}}{(2^n+\mu)^{\mu}} M_{\infty}(f_n,1) \\
  & & \leq  c_1 \sup_n\frac{\mu^{\mu}}{\mu_1^{\mu_1}} \frac{(2^n+\mu_1)^{\mu_1}}{(2^n+\mu)^{\mu}}  \frac{\mu_1^{\mu_1}}{(2^n+\mu_1)^{\mu_1}} M_{\infty}(f_n,r_{\mu_1,n}) \\
   & &  \leq
 c_1   \sup_n \delta_n ||f_n||_{\mu_1} \\
   & &     =
    c_1 \sup_n \delta_n || (P_{2^{n+1}-1}-P_{2^n-1})f ||_{\mu_1} \\
   & &     \leq
    c_1c_2 \sup_n \delta_n n ||f||_{\mu_1}
      \end{eqnarray*}
  where
  \[ \delta_n =\frac{\mu^{\mu}}{\mu_1^{\mu_1}} \frac{(2^n+\mu_1)^{\mu_1}}{(2^n+\mu)^{\mu}}  \]
  and $c_1$, $c_2$ are universal constants.
  	Since $\mu > \mu_1$ we obtain $ \sup_n \delta_n n < \infty$. This proves part (i).

   On the other hand, with Lemma \ref{estimates} and (3.3) applied to $\mu_2$ we obtain
   	 \begin{eqnarray*}
   	 ||f||_{\mu_2}  & \leq &  \sum_{n=0}^{\infty} ||f_n||_{\mu_2} \\
   	                & \leq & \sum_{n=0}^{\infty}   (1-r_{\mu_2,n})^{\mu_2}  M_{\infty}(f_n,1)     \\
   	                & \leq &  c    \sum_{n=0}^{\infty} \frac{\mu_2^{\mu_2}}{\mu^{\mu}} \frac{(2^n+\mu)^{\mu}}{(2^n+\mu_2)^{\mu_2}} \frac{\mu^{\mu}}{(2^n+\mu)^{\mu}}n^2
\sup_{2^n \leq j < 2^{n+1}}|f_n(e^{i 2 \pi j/2^n})| \\
                    & \leq & d \sup_n\frac{\mu^{\mu}}{(2^n+\mu)^{\mu}}
                    \sup_{2^n \leq j < 2^{n+1}}|f_n(e^{i 2 \pi j/2^n})|
  	                \end{eqnarray*}
                  where
 \[ d =c    \sum_{n=0}^{\infty} \frac{\mu_2^{\mu_2}}{\mu^{\mu}}\frac{(2^n+\mu)^{\mu}}{(2^n+\mu_2)^{\mu_2}}n^2. \]
 	Since $\mu_2 > \mu$ this series converges.

On account that dim $\{f_n : f \in A^{-\gamma}_+ \} = 2^n=$ number of the elements $\exp(i2 \pi j/2^n)$ if $j=2^n, \ldots, 2^{n+1}-1$, given $x=(x_j)$, the polynomials $f_n$ with  $f_n( e^{i2 \pi j/2^n})= x_j$ if $2^n \leq j \leq 2^{n+1}-1$ are uniquely defined. Consequently, the estimates above imply statement (ii).
 	
  The proof of Lemma \ref{technical} is now complete.

\bigskip

\noindent \textbf{Acknowledgements.} The research of Bonet was partially
supported by the project  MTM2016-76647-P. The research of Taskinen was
partially supported by the V\"ais\"al\"a Foundation of the Finnish Academy
of Sciences and Letters.

%%%%%%%%%%%%%%%%%%%%%%%%%%%%%%%%%%%%%%%%%%%%%%%%%%%%%%%%%%%%%%%%%%%%%%%%%
%%%%%%%%%%%%%%%%%%%%%%%%%%%%%%%%%%%%%%%%%%%%%%%%%%%%%%%%%%%%%%%%%%%%%%%%%
%%% Bibliography
%%%%%%%%%%%%%%%%%%%%%%%%%%%%%%%%%%%%%%%%%%%%%%%%%%%%%%%%%%%%%%%%%%%%%%%%%
%%%%%%%%%%%%%%%%%%%%%%%%%%%%%%%%%%%%%%%%%%%%%%%%%%%%%%%%%%%%%%%%%%%%%%%%%

\bibliographystyle{amsplain}

\noindent \textbf{Authors' addresses:}%
\vspace{\baselineskip}%

Jos\'e Bonet: Instituto Universitario de Matem\'{a}tica Pura y Aplicada IUMPA,
Universitat Polit\`{e}cnica de Val\`{e}ncia,  E-46071 Valencia, Spain

email: jbonet@mat.upv.es \\

Wolfgang Lusky: FB 17 Mathematik und Informatik, Universit\"at Paderborn, D-33098 Paderborn, Germany.

email: lusky@uni-paderborn.de \\

Jari Taskinen: Department of Mathematics and Statistics, P.O. Box 68,
University of Helsinki, 00014 Helsinki, Finland.

email: jari.taskinen@helsinki.fi

\end{document}